\documentclass{article}
\usepackage[utf8]{inputenc}
\usepackage{amsmath,amsfonts,amsthm}
\usepackage{graphicx}
\usepackage{hyperref}

\newtheorem{proposition}{Proposition}
\newtheorem{corollary}{Corollary}[proposition]
\newtheorem{assumption}{Assumption}[section]
\newtheorem{example}{Example}[section]

\title{Exact Dynamic Programming for\\Positive Systems with Linear Optimal Cost}
\author{Yuchao Li\thanks{Y. Li was with the Division of Decision and Control Systems, KTH Royal Institute of Technology, when part of this research was conducted. He is now with the School of Computing and Augmented Intelligence, Arizona State University. \href{yuchaoli@asu.edu}{yuchaoli@asu.edu}, \href{yuchao@kth.se}{yuchao@kth.se}}\and Anders Rantzer\thanks{A. Rantzer is a member of the Excellence Center ELLIIT and Wallenberg AI, Autonomous Systems and Software Program (WASP). Financial support was received from the European Research Council (Advanced Grant 834142). \href{anders.rantzer@control.lth.se}{anders.rantzer@control.lth.se}}}
\date{September 2023}

\begin{document}

\maketitle

\begin{abstract}
   Recent work \cite{rantzer2022explicit} formulated a class of optimal control problems involving positive linear systems, linear stage costs, and elementwise constraints on control. It was shown that the problem admits linear optimal cost and the associated Bellman's equation can be characterized by a finite-dimensional nonlinear equation, which is solved by linear programming. In this work, we report exact dynamic programming (DP) theories for the same class of problems. Moreover, we extend the results to a related class of problems where the norms of control are bounded while the optimal costs remain linear. In both cases, we provide conditions under which the solutions are unique, investigate properties of the optimal policies, study the convergence of value iteration, policy iteration, and optimistic policy iteration applied to such problems, and analyze the boundedness of the solution to the associated optimization programs. Apart from a form of the Frobenius-Perron theorem, the majority of our results are built upon generic DP theory applicable to problems involving nonnegative stage costs.     
\end{abstract}

\section{Introduction}
The study of generic nonnegative cost optimal control problems dates back to the thesis and paper \cite{strauch1966negative}, following the seminal earlier research of \cite{blackwell1967positive,blackwell1965discounted}. Owing to its wide range of applications, the followup work has been voluminous and comprehensive. In particular, a unified dynamic programming (DP) framework stemming from the nonnegativity of costs was introduced by Bertsekas in \cite{bertsekas1975monotone,bertsekas1977monotone}. It incorporates and extends many earlier results developed in a variety of contexts, facilitates the analysis of value iteration (VI), policy iteration (PI), and optimistic PI algorithms, and applies to a wide range of problems.

Despite the fruitful development of the general DP theory, when the optimal control problems involve continuous quantities taking values from some Euclidean spaces, the exact solutions are often intractable. A well-known exception is the linear quadratic problem developed by Kalman \cite{kalman1960contributions}, which predated the work \cite{strauch1966negative}. The essential property that enables the exact solution is that the optimal cost is quadratic. Due to this reason, it also admits extensions where computing exact solutions remains practical; see, e.g., \cite{barratt2021stochastic}. The recent work \cite{rantzer2022explicit} formulates another class of optimal control problems, where favorable results can be obtained. The problems involve positive linear systems, linear stage costs, and linear constraints. It is shown that the optimal cost function is linear and the solution can be characterized by a finite-dimensional nonlinear equation with respect to the cost parameters, which is solved by linear programming. Its development relies in part on the nonnegativity of stage cost. However, the full-fledged DP theory pioneered in \cite{bertsekas1975monotone} was not brought to bear.

Apart from the general DP theory, literature on the positive linear systems provides additional tools for the study of the optimal control problem posed in \cite{rantzer2022explicit}. An early classic in this subject is \cite{luenberger1979introduction}, where many related results are synthesized. Within this context, the cornerstone is the Frobenius-Perron theorem, which characterizes the eigenvalues of matrices with nonnegative elements. As a result, it is an essential tool for analyzing the stability of positive linear systems. For recent developments in the control of such systems, see \cite{rantzer2021scalable}.

In this work, we develop the DP theory for the optimal control problems formulated in \cite{rantzer2022explicit} and extend it to a closely related variation. For both classes of problems, our theory relies on the generic DP theory for nonnegative cost problems and a form of the Frobenius-Perron theorem. The contributions are as follows:
\begin{itemize}
    \item[(a)] Characterization of linearity of optimal cost functions and conditions under which the solutions are unique (Prop.~\ref{prop:unique} and Prop.~\ref{prop:unique_norm}).
    \item[(d)] Stability analysis of the optimal policies (Prop.~\ref{prop:optimal_policy} and Prop.~\ref{prop:optimal_policy_norm}).
    \item[(c)] Convergence analysis of VI, PI, and optimistic PI applied to such problems (Props.~\ref{prop:vi}, \ref{prop:pi}, \ref{prop:optimistic_pi} and Props.~\ref{prop:vi_norm}, \ref{prop:pi_norm}, \ref{prop:optimistic_pi_norm}).
    \item[(d)] Conditions under which the associated optimization programs are bounded (Prop.~\ref{prop:linear} and Prop.~\ref{prop:optimization}).
\end{itemize}

The results reported in this work incorporate and extend those given in \cite{rantzer2022explicit}. Although theoretical in nature, these results can be used as starting points for developing approximate solution schemes for similar problems with additional constraints and/or modifications.

\subsection*{Notations}
We use $\Re$ to denote the real line and $\Re^n$ to denote the $n$ dimensional Euclidean space. Its positive orthant is denoted as $\Re^n_+$. We denote by $\Re^{m\times n}$ the set of $m$ by $n$ matrices. By vector, we mean column vector. Row vectors are specified explicitly. We use $0$ to denote the scalar zero or the vector/matrix with all zero elements. It will be clear from the context the meaning of $0$. The notation $1$ is used similarly. For a matrix $M$, we use $|M|$ to denote the matrix obtained by replacing the elements of $M$ with their absolute values. If $M=|M|$ ($M=-|M|$), it is denoted by $M\geq 0$ ($M\leq 0$). If $M\geq 0$ ($M\leq 0$) and its elements are nonzero, we write $M>0$ ($M<0$, respectively). For matrices $M,N$ so that $M-N\geq0$ ($M-N\leq0$), we write $M\geq N$ ($M\leq N$, respectively). The notations $M>N$ and $M<N$ are defined similarly.\footnote{Be aware of the difference in the use of inequality notations in the literature. For example, in \cite[p.~7]{farina2000positive}, matrix $M$ with all positive elements are denoted as $M\gg0$, and $M>0$ means that the elements of $M$ are nonnegative and at least one is positive. Since our analysis does not directly rely on at least one element of matrices of vectors being positive (but indirectly), we do not adopt this convention.} Column and row vectors are special forms of matrices, so the function $|\cdot|$ and relations $\leq$, $\geq$, $<$, $>$ apply to them. Given a vector $v$, we use $\Vert v\Vert$ to denote its norm, and $\Vert v\Vert_*$ to denote the corresponding dual norm.\footnote{The theory developed in this work applies to any vector norm, so we would leave the form of the norm unspecified. Occasionally, we would make comments on the results when particular norms are used.} 

\section{Problem Formulations}\label{sec:formulation}
We now formulate the two types of optimal control problems studied in this work, both of which involve positive linear systems and linear cost. The first class, referred to as the \emph{absolute value bound problem}, was introduced in \cite{rantzer2022explicit}. It was shown that its solution can be characterized by a finite-dimensional nonlinear equation, and thus can be solved exactly by addressing a finite-dimensional linear program. However, the uniqueness of the solution to this nonlinear equation, as well as the convergence of various classical DP algorithms applied to this problem, was left unanswered. A closely related problem referred to as the \emph{norm bound problem}, is new and will be introduced afterwards. 

The absolute value bound problem is defined as for every $x_0\in \Re_+^n$, solve\footnote{Under our conditions to be introduced shortly, the infimum can be attained. Therefore, we use $\min$ in place of $\inf$. The same holds for the norm bound problem.} 
\begin{align*}
	\min_{\{u_k\}_{k=0}^{\infty}}& \quad \sum_{k=0}^{\infty} s'x_k+r'u_k\\
	\mathrm{s.\,t.} & \quad x_{k+1}=Ax_k+Bu_k,\,k = 0,...,\\
	& \quad |u_k|\leq Ex_k,\,k = 0,...,
\end{align*}
where $s$, $r$, $A$, $B$, and $E$ are given quantities. In particular, $s\in \Re^n$, $r\in \Re^m$, $A$ is an $n$ by $n$ matrix, and $B=[b_1 \dots b_m]$ is an $n$ by $m$ matrix with $b_i$ being column vectors. Subsequently, we also write $B=[B_1'\dots B_n']'$ with $B_i$ being row vectors and the prime denotes transposition. The quantity $E=[E_1' \dots E_m']'$ is an $m$ by $n$ matrix with $E_i$ being row vectors. To ensure the quantities $x_k$, $k=1,\dots$, all stay within $\Re^n_+$, and $s'x_k+r'u_k$, $k=1,\dots$, are nonnegative, it is imposed in \cite{rantzer2022explicit} that the elements of $E$ are nonnegative and
\begin{equation}
    \label{eq:anders22}
    A\geq |B|E,\quad s\geq E'|r|.
\end{equation}
To obtain the desired solution property, we require an additional condition 
$$(s'- |r'|E)\sum_{i=0}^{n-1}(A-|B|E)^i>0.$$
The justification of the conditions will be given in Section~\ref{sec:ab_bound_prob}. In particular, the essential feature that simplifies this problem is that the optimal value of the problem is linear in $x_0$, which was given in \cite{rantzer2022explicit}. We will show that under the conditions introduced above, the optimal cost at $x_0$, denoted by $J^*(x_0)$, is of the form 
$$J^*(x_0)=x_0'p^*,$$ 
where $p^*$ is the unique solution to an associated nonlinear equation, thus enabling various exact solution approaches.  

Similarly, the norm bound problem considers for every $x_0\in \Re^n_+$, the optimization problem
\begin{align*}
	\min_{\{u_k\}_{k=0}^{\infty}}& \quad \sum_{k=0}^{\infty} s'x_k+r'u_k\\
	\mathrm{s.\,t.} & \quad x_{k+1}=Ax_k+Bu_k,\,k = 0,...,\\
	& \quad \Vert u_k\Vert\leq Nx_k,\,k = 0,...,
\end{align*}
where $s$, $r$, $A$, and $B$ are given quantities as in the absolute value bound problem, and $N$ is an $n$ dimensional row vector, which is also given. The problem is well-posed, provided that the elements of $N$ are nonnegative,
$$A\geq [\Vert B_1'\Vert_*\dots \Vert B_n'\Vert_*]'N,\quad s\geq N'\Vert r\Vert_*,$$
and 
$$(s'-\Vert r\Vert_*N)\sum_{i=0}^{n-1}\big(A- [\Vert B_1'\Vert_*\dots \Vert B_n'\Vert_*]'N\big)^i>0.$$
Further discussions on these conditions are given in Section~\ref{sec:norm}. Similar to the absolute value bound case, we will show that under the conditions introduced above, the optimal value is linear in $x_0$, and the linear parameter that defines the optimal values is the unique solution to an associated nonlinear equation. 

In the next section, we will show that these two classes of problems are special forms of the general nonnegative cost problem with a favorable positive linear structure. With this insight, we can develop the exact DP theory for both cases by bringing to bear the related DP theory and a classical result for positive linear systems.

\section{Background}
General deterministic discrete-time optimal control problems involve the system
\begin{equation}
\label{eq:dynamics}
    x_{k+1}=f(x_k,u_k),\quad k=0,\,1,\,\dots,
\end{equation}
where $x_k$ and $u_k$ are state and control at stage $k$, which belong to state and control spaces $X$ and $U$, respectively, and $f$ maps $X\times U$ to $X$. The control $u_k$ must be chosen from a nonempty constraint set $U(x_k)\subset U$ that may depend on $x_k$. The cost for the $k$th stage, denoted $g(x_k,u_k)$, is assumed nonnegative and may possibly take the value infinity:
\begin{equation}
\label{eq:cost_nonneg}
    0\leq g(x_k,u_k)\leq\infty,\quad x_k\in X,\,u_k\in U(x_k).
\end{equation}
By allowing an infinite value of $g(x,u)$ we can implicitly introduce state and control constraints: a pair $(x,u)$ is infeasible if $g(x,u)=\infty$. We are interested in feedback policies of the form $\pi=\{\mu_0,\mu_1,\dots\}$, where $\mu_k$ is a function mapping every $x\in X$ into the control $\mu_k(x)\in U(x)$. The set of all policies is denoted as $\Pi$. Policies of the form $\pi=\{\mu,\mu,\dots\}$ are called \emph{stationary}, and for convenience, when confusion cannot arise, will be denoted as $\mu$.

Given an initial state $x_0$, a policy $\pi=\{\mu_0,\mu_1,\dots\}$ when applied to system \eqref{eq:dynamics}, generates a unique sequence of state control pairs $\big(x_k,\mu_k(x_k)\big),$ $k=0,1,\dots$, with cost 
$$J_{\pi}(x_0)=\lim_{k\to\infty}\sum_{t=0}^kg\big(x_t,\mu_t(x_t)\big),\quad  x_0\in X.$$
We view $J_\pi$ as a function over $X$ that takes values in $[0,\infty]$. We refer to it as the cost function of $\pi$. For a stationary policy $\mu$, the corresponding cost function is denoted by $J_\mu$. The optimal cost function is defined as
$$J^*(x)=\inf_{\pi\in\Pi}J_{\pi}(x),\quad x\in X,$$
and a policy $\pi^*$ is said to be optimal if it attains the minimum of $J_\pi(x)$ for all $x\in X$, i.e., 
$$J_{\pi^*}(x)=\inf_{\pi\in \Pi}J_\pi(x)=J^*(x),\quad \forall x\in X.$$

In the context of DP, one hopes to prove that the optimal cost function $J^*$ satisfies Bellman's equation:
\begin{equation}
    \label{eq:bellman}
    J^*(x)=\inf_{u\in U(x)}\big\{g(x,u)+J^*\big(f(x,u)\big)\big\},\quad \forall x\in X,
\end{equation}
and that a stationary optimal policy exists. The classical solution algorithms for general optimal control problems include VI and PI, both traced back to Bellman \cite[Section~12]{bellman1954theory}, \cite[p.~88]{bellman1957dynamic}. For the nonnegative cost problem, the VI starts with some nonnegative function $J_0:X\mapsto[0,\infty]$, and generates a sequence of functions $\{J_k\}$ according to 
\begin{equation}
    \label{eq:vi}
    J_{k+1}(x)=\inf_{u\in U(x)}\big\{g(x,u)+J_k\big(f(x,u)\big)\big\},\quad \forall x\in X.
\end{equation}

The PI algorithm starts from a stationary policy $\mu^0$, and generates a sequence of stationary policies $\{\mu^k\}$ via a sequence of policy evaluations to obtain $J_{\mu^k}$ from the equation 
\begin{equation}
    \label{eq:policy_ev}
    J_{\mu^k}(x)=g\big(x,\mu^k(x)\big)+J_{\mu^k}\big(f\big(x,\mu^k(x)\big)\big),\quad x\in X,
\end{equation}
interleaved with policy improvements to obtain $\mu^{k+1}$ from $J_{\mu^k}$ via 
\begin{equation}
    \label{eq:policy_im}
    \mu^{k+1}(x)\in \arg\min_{u\in U(x)}\big\{g(x,u)+J_{\mu^k}\big(f(x,u)\big)\big\},\quad x\in X,
\end{equation}
where we have assumed that the relevant minimums can be attained.

There are some classical results that hold for this class of problems. Note that the convergence of sequence, and equalities/inequalities between functions are meant pointwise.

\begin{proposition}\label{prop:classical}
Let the cost nonnegativity condition \eqref{eq:cost_nonneg} hold.
\begin{itemize}
    \item[(a)] $J^*$ satisfies Bellman's equation \eqref{eq:bellman}, and if $\hat{J}:X\mapsto[0,\infty]$ satisfies
    $$\inf_{u\in U(x)}\big\{g(x,u)+\hat{J}\big(f(x,u)\big)\big\}\leq \hat{J}(x),\quad \forall x\in X,$$
    then $J^*\leq \hat{J}$.
    \item[(b)] For all stationary policies $\mu$ we have 
    $$J_{\mu}(x)=g\big(x,\mu(x)\big)+J_{\mu}\Big(f\big(x,\mu(x)\big)\Big),\quad \forall x\in X.$$
    Moreover, if $\hat{J}:X\mapsto[0,\infty]$ satisfies
    $$g\big(x,\mu(x)\big)+\hat{J}\Big(f\big(x,\mu(x)\big)\Big)\leq \hat{J}(x),\quad \forall x\in X,$$
    then $J_\mu\leq \hat{J}$.
    \item[(c)] A stationary policy $\mu^*$ is optimal if and only if 
    $$\mu^*(x)\in\arg\min_{u\in U(x)}\big\{g(x,u)+J^*\big(f(x,u)\big)\big\},\quad \forall x\in X.$$
    \item[(d)] If a function $J$ satisfies 
    $$J^*\leq J\leq cJ^*$$
    for some scalar $c>1$, then the sequence $\{J_k\}$ generated by VI with $J_0=J$ converges to $J^*$, i.e., $J_k\to J^*$.
\end{itemize}
\end{proposition}
The second half of Prop.~\ref{prop:classical}(a) is first given in \cite[Prop.~5.2]{bertsekas1978stochastic}, which may be viewed as the discrete-time counterpart of the control Lyapunov function introduced in \cite{sontag1983lyapunov} without additional conditions on the stage cost. Prop.~\ref{prop:classical}(b) is a special case of (a). Prop.~\ref{prop:classical}(c) is given as \cite[Prop.~5.4]{bertsekas1978stochastic}, and Prop.~\ref{prop:classical}(d) is proved as \cite[Theorem~5.1]{yu2015mixed} within the context of stochastic problems with additive costs, where the associated measurability issue is also addressed. Note that those results can be extended to problems far beyond the scope considered here, including stochastic problems with additive costs [valid after suitable modifications of Prop.~\ref{prop:classical}(a)-(d)], stochastic problems with multiplicative costs [for Prop.~\ref{prop:classical}(a)-(c)], and minimax control [for Prop.~\ref{prop:classical}(a)-(c)].

Having the general nonnegative cost problems in mind, one can see that the problems formulated in Section~\ref{sec:formulation} are their special cases. Through the lens of DP theory, both classes of problems involve a common form of system dynamics and stage cost. The difference lies in the form of control constraint sets. The common system dynamics are
\begin{equation}
    \label{eq:dp_dynamics}
    f(x,u)=Ax+Bu,
\end{equation}
and the stage cost is given as
\begin{equation}
    \label{eq:dp_cost}
    g(x,u)=s'x+r'u.
\end{equation}
In the absolute value bound problem, the control constraint set $U(x)$ is given as
\begin{equation}
    \label{eq:control_ab_v_b}
    U(x)=\{u\in \Re^m\,|\,|u|\leq Ex\},
\end{equation}
while the control constraint set $U(x)$ is given as
\begin{equation}
    \label{eq:control_norm_b}
    U(x)=\{u\in \Re^m\,|\,\Vert u\Vert\leq Nx\},
\end{equation}
in the norm bound problem.

For the absolute value bound problems, \cite{rantzer2022explicit} introduced suitable conditions [$E$ being nonnegative and the first inequality in \eqref{eq:anders22} being true] to make the states remain within $\Re^n_+$. Systems of this form are commonly referred to as `positive linear systems.' For the study of positive linear systems, the single most important result is known as the Frobenius-Perron theorem. It admits several variants and the following form is given as \cite[Theorem 3, Section 6.2]{luenberger1979introduction}, within the context of positive linear systems.
\begin{proposition}\label{prop:f_p}
    Let $M$ be an $n$ by $n$ matrix such that $M\geq 0$. Then there exists some scalar $\lambda\geq0$ and some nonzero vector $x\in \Re_+^n$ such that $Mx=\lambda x$ and for all other eigenvalues $\rho$ of $M$, $|\rho|\leq \lambda$.
\end{proposition}

In the following sections, we will first develop in detail the DP theory for absolute value bound problem. This is followed by the DP theory for the norm bound problem. The proof of the results will be provided in the subsequent section.

\section{Dynamic Programming for the Absolute Value Bound Problem}\label{sec:ab_bound_prob}
For the absolute value bound problem, we take the perspective of the general DP theory and summarize the related conditions discussed in Section~\ref{sec:formulation} as the following assumption, which involves a correction of the conditions given in \cite[Theorem~1]{rantzer2022explicit}. We will show through examples (Examples~\ref{eg:bellman_nonunique} and \ref{eg:scalar}) that the additional condition, compared with those stated in \cite[Theorem~1]{rantzer2022explicit}, is necessary in order to ensure the uniqueness of the solution to Bellman's equation. \emph{The assumption remains valid throughout this section and thus is omitted in the subsequent theoretical statements.} Recall that inequalities between vectors and matrices mean elementwise comparison.
\begin{assumption}\label{asm:1}
The system dynamics and the stage cost are given as \eqref{eq:dp_dynamics} and \eqref{eq:dp_cost}, respectively. The control constraint sets $U(x)$ take the form \eqref{eq:control_ab_v_b}, where the elements of $E$ are nonnegative. Moreover,
\begin{equation}
    \label{eq:asum}
    A\geq |B|E,\quad s\geq E'|r|,
\end{equation}
and 
\begin{equation}
    \label{eq:asum_ob}
    (s'- |r'|E)\sum_{i=0}^{n-1}(A-|B|E)^i>0.
\end{equation}
\end{assumption}

The direct consequences of the assumption are as follows:
\begin{itemize}
    \item[(1)] The positive orthant of $\Re^n$ can be used as the state space $X$ and the problem is well-posed in the sense that
    $$f(x,u)\in X,\quad \forall x\in X,\,u\in U(x).$$
    \item[(2)] The stage cost $g(x,u)$ is nonnegative for all $x\in X$. In particular, $g(0,0)=0$. The inequality \eqref{eq:asum_ob} is the \emph{observability} condition, ensuring states other than zero have a positive cost under arbitrary policies.\footnote{A further discussion on this property is peripheral to the main point of the following development and is thus deferred toward the appendix.}
\end{itemize}
Therefore, the problem studied here falls into the general nonnegative cost problem, and the classical results given in Prop.~\ref{prop:classical} can be brought to bear.

Now we are ready to state the main results for this class of problems. Their proofs are deferred until Section~\ref{sec:proof}.

\begin{proposition}
    \label{prop:unique}
    The following two statements are equivalent:
    \begin{itemize}
    \item[(i)] $J^*(x)$ is finite for all $x\in X$.
    \item[(ii)] There exists some $p\in \Re^n_+$ such that 
    \begin{equation}
        \label{eq:bellman_pos}
        p=s+A'p-E'|r+B'p|.
    \end{equation}
    Moreover, the solution $p^*$ is unique within $\Re^n_+$ and $J^*(x)=x'p^*$ for all $x\in X$.
\end{itemize}
\end{proposition}

Note that the strict inequality in \eqref{eq:asum_ob} is needed to avoid some pathologies in Bellman's equation \eqref{eq:bellman_pos}. The following is an example where the uniqueness of the solution does not hold due to the equality $s=E'|r|$.

\begin{example}\label{eg:bellman_nonunique}
    Let $r=|r|$. Suppose further that $A-BE=I$ with $I$ being the identity matrix of suitable dimension, and $s=E'r$. This is true if we set $m=n$, $A=2I$, $B=E=I$, and $s=r$. As a result, one can verify that $J^*(x)=0$ for all $x$ so that we may write $J^*(x)=x'p^*$ with $p^*=[0\;0]'$. To see this, consider the policy $\mu(x)=Lx$ with $L=-I$ so that $|\mu(x)|=|x|\leq E'x$. The stage cost under this policy is $g\big(x,\mu(x)\big)=s'x-r'x=0$ for all $x\in X$ so that $J_\mu=J^*$. However, for $p\in \Re^n_+$, the right-hand side of equation \eqref{eq:bellman_pos} takes the form 
    $$s+A'p-E'|r+B'p|=s+2p-r-p=p.$$
    As a result, every $p\in \Re^n_+$ fulfills the equation \eqref{eq:bellman_pos} and its solution is not unique. Still, $p^*$ is the smallest solution among them within $\Re^n_+$, which is consistent with Prop.~\ref{prop:classical}(a). However, the absence of uniqueness of the solution renders major difficulties for the algorithms developed later.
\end{example}

On the other hand, under Assumption~\ref{asm:1}, the pathology discussed above can be avoided.

\begin{example}\label{eg:scalar}
Consider a scalar system with $A=1$, $B=0.5$, $s=1.5$, $r=-1$, and $E=1$. Then Assumption~\ref{asm:1} holds. The plot of Bellman's equation is given in Fig.~\ref{eg:scalar}.

\begin{figure}[ht]
    \centering
    \includegraphics[width=0.9\linewidth]{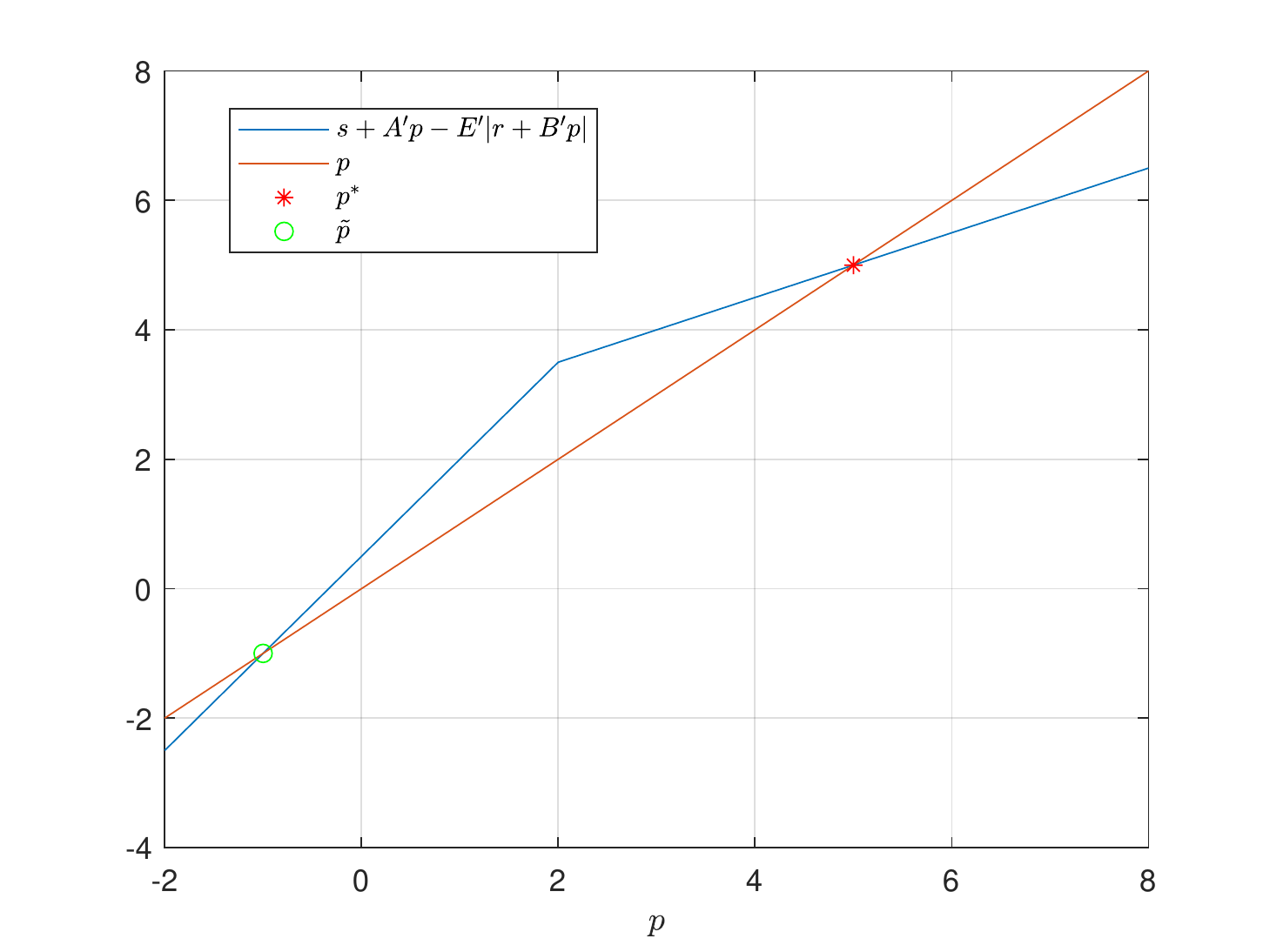}
    \caption{The Bellman curve and the solutions to the Bellman's equation for Example~\ref{eg:scalar}. Note that $p^*$ is the only nonnegative solution, while the other solution, denoted as $\Tilde{p}$, is negative.}
    \label{fig:scalar}
\end{figure}
\end{example}

When $p^*$ is finite, in view of Prop.~\ref{prop:classical}(c) and Prop.~\ref{prop:unique}, we have a stationary linear optimal policy $\mu^*(x)=\Bar{L}^*x$ with
\begin{equation}
    \label{eq:optimal_policy_bar}
    \Bar{L}^*=-\begin{bmatrix}
\text{sign}(r_1+b_1'p^*)E_1\\
\vdots\\
\text{sign}(r_m+b_m'p^*)E_m
\end{bmatrix},
\end{equation}
where $\text{sign}:\Re\mapsto\{-1,1\}$ takes the value $1$ if the argument is nonnegative and $-1$ otherwise, $r_i$ is the $i$th element of $r$, and recall that $E_i$ is the $i$th row of $E$. In fact, there exists a set of stationary optimal policies that are linear, and their linear parameters form the set $\mathcal{L}_e^*$, which is defined as
\begin{equation}
    \label{eq:optimal_policy_set}
    \mathcal{L}_e^*=\bigg\{L\in \Re^{m\times n}\,\Big|\,|L|\leq E,\;\sum_{i=1}^n(r_i+b_i'p^*)\Vert L_i-\Bar{L}^*_i\Vert=0\bigg\},
\end{equation}
where $L_i$ and $L^*_i$ are $i$th rows of $L$ and $\Bar{L}^*$, respectively. Besides, the positivity of the system states allows further characterization of those policies.
For an $n$ by $n$ matrix $M$, we say it is \emph{(Schur) stable} if its eigenvalues are strictly within the unit circle. Now we are ready to characterize the optimal policy.
\begin{proposition}
    \label{prop:optimal_policy}
    Suppose that $J^*(x)$ is finite for all $x\in X$. Then the set $\mathcal{L}_e^*$ given in \eqref{eq:optimal_policy_set} is well-defined and nonempty. Moreover, for every $L^*\in \mathcal{L}_e^*$, the stationary linear policy $L^*x$ is optimal, i.e., $\mu^*(x)=L^*x$, and $(A+BL^*)$ is stable. 
\end{proposition}
Note that there may be some optimal policy that is nonlinear. To see this, assume $r_i+b_i'p^*=0$. Then we may define the $i$th term of $u$ as a nonlinear function of $x$, as long as its magnitude is bounded by $Ex$.

Adapting the classical stabilizability concept (see, e.g., \cite[p.~214]{sontag2013mathematical}, \cite[p.~121]{bertsekas2017dynamic}), we say the system is \emph{stablizable} if there exists some $L\in \Re^{m\times n}$ such that $\mu(x)=Lx$ is well-defined, i.e., $|Lx|\leq Ex$ for all $x$, and $(A+BL)$ is stable. Then a direct application of Prop.~\ref{prop:optimal_policy} is as follows.

\begin{corollary}
    \label{coro:stabilizable}
    The following two statements are equivalent:
    \begin{itemize}
    \item[(i)] $J^*(x)$ is finite for all $x\in X$.
    \item[(ii)] The system is stabilizable.
\end{itemize}
\end{corollary}

For this problem, the VI algorithm takes the form\footnote{For a detailed derivation of the following formula, see \cite{rantzer2022explicit}.} 
\begin{equation}
    \label{eq:vi_pos}
    p_{k+1}=s+A'p_k-E'|r+B'p_k|.
\end{equation}
In what follows, we show that the VI algorithm \eqref{eq:vi_pos} converges to $p^*$ starting from arbitrary $p_0\in \Re^n_+$.
\begin{proposition}
    \label{prop:vi}
    Suppose that $J^*(x)$ is finite for all $x\in X$. Then the sequence $\{p_k\}$ generated by VI \eqref{eq:vi_pos} with $p_0\in \Re^n_+$ converges to $p^*$.
\end{proposition}

Assume that $J^*(x)$ is finite for all $x\in X$. Then in view of Prop.~\ref{prop:optimal_policy}, there exists some $\mu^0(x)=L_0x$ such that $(A+BL_0)$ is stable (such as the elements of $\mathcal{L}_e^*$, which is nonempty assured by Prop.~\ref{prop:optimal_policy}). When applying PI algorithm, at a typical iteration $k$, given the policy $\mu^k(x)=L_kx$, the policy evaluation step takes the form 
\begin{equation}
    \label{eq:pi_eva_pos}
    p_{\mu^k}'=(s+L_k'r)'(I-A-BL_k)^{-1}.
\end{equation}
The improved policy $\mu^{k+1}$ takes the form $\mu^{k+1}(x)=L_{k+1}x$ with $L_{k+1}$ given as 
\begin{equation}
    \label{eq:pi_imp_pos}
    L_{k+1}=-\begin{bmatrix}
\text{sign}(r_1+b_1'p_{\mu^k})E_1\\
\vdots\\
\text{sign}(r_m+b_m'p_{\mu^k})E_m
\end{bmatrix}.
\end{equation}
Note that when $r_i+p_{\mu^k}'b_i=0$, the respective row of $L_{k+1}$ can be $\alpha E_i$ where $\alpha$ is an arbitrary constant within the interval $[-1,1]$. However, it is beneficial to restrict our attention to a smaller class of policies that contains an optimal policy. Next, we show that the above PI algorithm is well-posed and the generated cost and control pairs converge to the optimal cost and the optimal stationary policy, respectively.

\begin{proposition}
    \label{prop:pi}
    Suppose that $J^*(x)$ is finite for all $x\in X$. Then there exists some $L_0$ such that $\mu^0(x)=L_0x$ is well-defined, i.e., $L_0x\in U(x)$ for all $x\in X$, and $(A+BL_0)$ is stable. Moreover, starting from $\mu^0(x)=L_0x$, the PI algorithm \eqref{eq:pi_eva_pos} and \eqref{eq:pi_imp_pos} is well-defined, i.e., $L_kx\in U(x)$ for all $x\in X$, and $(A+BL_k)$ are stable, $k=1,2,\dots$, and $p_{\mu^k}= p^*$, $L_k=L^*$ for all $k\geq 2^m$.
\end{proposition}

The PI algorithm also admits an optimistic variant. In particular, let $\{\ell_k\}$ be a sequence of positive integers. The optimistic PI starts with some $p_0$ such that 
\begin{equation}
    \label{eq:op_pi_initial_pos}
    p_0\geq s+A'p_0-E'|r+B'p_0|.
\end{equation}
At a typical iteration $k$, given $p_k$, it defines a policy $\mu^k(x)=L_k(x)$ by
\begin{equation}
    \label{eq:op_pi_policy}
    L_k=-\begin{bmatrix}
\text{sign}(r_1+b_1'p_k)E_1\\
\vdots\\
\text{sign}(r_m+b_m'p_k)E_m
\end{bmatrix},
\end{equation}
and obtain $p_{k+1}$ by
\begin{equation}
    \label{eq:op_pi_cost}
    p_{k+1}'=p_k'(A+BL_k)^{\ell_k}+(s+L_k'r)'\sum_{i=0}^{\ell_k-1}(A+BL_k)^i.
\end{equation}
In what follows, we will show that the optimistic PI is well-posed and the generated $p_k$ converges to $p^*$.

\begin{proposition}
    \label{prop:optimistic_pi}
    Suppose that $J^*(x)$ is finite for all $x\in X$. Then there exists $p_0$ that satisfies the inequality \eqref{eq:op_pi_initial_pos}, and the optimistic PI \eqref{eq:op_pi_policy} and \eqref{eq:op_pi_cost} is well-defined, i.e., $L_kx\in U(x)$ for all $x\in X$, $k=0,1,\dots$. Moreover, $(A+BL_k)$ are stable, $k=0,1,\dots$, the generated $p_k$ converges to $p^*$, i.e., $p_k\to p^*$, and there exists some $\Bar{k}$ such that $L_k\in \mathcal{L}_e^*$ for all $k\geq \Bar{k}$.
\end{proposition}

The problem can also be solved by linear programming in view of the uniqueness result given in Prop.~\ref{prop:unique}. The linear program takes the form 

\begin{subequations}
\label{eq:linear}
    \begin{align}
	\max_{p,\gamma}& \quad 1'p\\
	\mathrm{s.\,t.} & \quad p=s+A'p-E'\gamma,\label{eq:linear_eq}\\
	& \quad  -\gamma\leq r+B'p\leq \gamma, \label{eq:linear_ineq}\\
	& \quad  p\in \Re^n_+,\;\gamma\in \Re^m_+,\label{eq:linear_range}
	\end{align}
\end{subequations} 
where $1$ denotes a vector of suitable dimension with all elements equal to scalar one. The linear program has been given in \cite[Theorem~1]{rantzer2022explicit} to find the optimal solution $J^*$. However, the pathological case indicated in Example~\ref{eg:bellman_nonunique} is not addressed there. Here we highlight the importance of the uniqueness of the solution by proving the following proposition. 

\begin{proposition}
    \label{prop:linear}
    The following two statements are equivalent:
    \begin{itemize}
    \item[(i)] $J^*(x)$ is finite for all $x\in X$.
    \item[(ii)] The linear program \eqref{eq:linear} admits a finite optimal value.
\end{itemize}
\end{proposition}

\section{Dynamic Programming for the Norm Bound Problems}\label{sec:norm}
In the preceding section, we developed DP theory for optimal control of positive linear systems where the elements of the control are bounded by linear functions of states. In this section, we study optimal control problems where the norms of the control are bounded. \emph{For this class of problems, we summarize the related conditions given in Section~\ref{sec:formulation} as the following standing assumption, which is omitted in all theoretical results of this section.} 
\begin{assumption}\label{asm:2}
The system dynamics and the stage cost are given as \eqref{eq:dp_dynamics} and \eqref{eq:dp_cost}, respectively. The control constraint sets $U(x)$ take the form \eqref{eq:control_norm_b}, where the elements of $N$ are nonnegative. Moreover,
\begin{equation}
    \label{eq:asum_norm}
    A\geq [\Vert B_1'\Vert_*\dots \Vert B_n'\Vert_*]'N,\quad s\geq N'\Vert r\Vert_*,
\end{equation}
and 
\begin{equation}
    \label{eq:asum_norm_ob}
    (s'-\Vert r\Vert_*N)\sum_{i=0}^{n-1}\big(A- [\Vert B_1'\Vert_*\dots \Vert B_n'\Vert_*]'N\big)^i>0.
\end{equation}
\end{assumption}

Under Assumption~\ref{asm:2}, the identical observations following Assumption~\ref{asm:1} can be made for the norm bound problem. We are now ready to state the main results for this class of problems. One may notice the similarity between those stated in this section and the preceding section.

\begin{proposition}
    \label{prop:unique_norm}
    The following two statements are equivalent:
    \begin{itemize}
    \item[(i)] $J^*(x)$ is finite for all $x\in X$.
    \item[(ii)] There exists some $p\in \Re^n_+$ such that 
    \begin{equation}
        \label{eq:bellman_pos_norm}
        p=s+A'p-N'\Vert r+B'p\Vert_*.
    \end{equation}
    Moreover, the solution $p^*$ is unique within $\Re^n_+$ and $J^*(x)=x'p^*$ for all $x\in X$.
\end{itemize}
\end{proposition}

As in the absolute value bound problem, the strict inequality in \eqref{eq:asum_norm_ob} is needed to ensure the uniqueness of the solution. As an illustration, we provide an example.

\begin{example}\label{eg:bellman_nonunique_norm}
Let us consider a scalar system with $A=0.9$, $B=1$, $s=1$, $r=-10$, and $N=0.1$. One can verify that the first inequality in Assumption~\ref{asm:2} holds while $s= N\Vert r\Vert_*$. Then the optimal cost is $J^*(x)=0$ for all $x$ with $p^*=0$. This is achieved by the policy $\mu(x)=0.1x$ so that the stage cost under this policy is identically zero. 

Let us now examine the right-hand side of \eqref{eq:bellman_pos_norm}, which takes the form
$$s+A'p-N'\Vert r+B'p\Vert_*=1+0.9p-0.1|-10+p|.$$
For $p\in [0,10]$, the expression can be written as $1+0.9p-0.1(10-9)=p$. As a result, every $p\in [0,10]$ is a nonnegative solution to the corresponding Bellman's equation. Similarly, for $p>10$, the right-hand side takes the form $2+0.8p$, which is strictly less than $p$ as $p>10$.
\end{example}

Similar to the preceding section, for the norm bound problem, if $J^*$ is finite, the optimal linear policies lead to stable matrices. However, the parameters cannot be specified by analytical formulas applicable to arbitrary norms. To describe those policies succinctly, we introduce a set-valued function $\mathcal{S}$ that maps vectors $v\in \Re^m$ to the power sets of $\Re^m$, given as\footnote{For readers familiar with convex analysis, the dual norm $\Vert\cdot\Vert_*$ is a closed proper convex function, and the set-valued function $\mathcal{S}$ defines its subdifferential, the elements of which are subgradients; see, e.g., \cite[p.~214]{rockafellar1997convex} and \cite[p.~182]{bertsekas2009convex}. As a result, the convergence of policies stated in Props.~\ref{prop:pi_norm} and \ref{prop:optimistic_pi_norm} can be established via continuity property of subgradients for closed proper convex functions; cf. \cite[Theorem~24.4]{rockafellar1997convex}. However, our proofs do not rely on such connections.}
\begin{equation}
    \label{eq:subdifferential}
    \mathcal{S}(v)=\begin{cases}
\{w\in \Re^m\,|\,\Vert w\Vert =1,\;w'v=\Vert v\Vert_*\}&\text{if }v\neq 0,\\
\{w\in \Re^m\,|\,\Vert w\Vert \leq1\}&\text{otherwise.}\end{cases}
\end{equation}
Then the set of parameters characterizing optimal linear policies form the set $\mathcal{L}_n^*$, which is defined as
\begin{equation}
    \label{eq:optimal_policy_set_norm}
    \mathcal{L}_n^*=\big\{L\in \Re^{m\times n}\,\big|\,L=-wN,\;w\in \mathcal{S}(r+B'p^*)\big\}.
\end{equation}
We have the following result regarding the optimal policies.
\begin{proposition}
    \label{prop:optimal_policy_norm}
    Suppose that $J^*(x)$ is finite for all $x\in X$. Then the set $\mathcal{L}_n^*$ given in \eqref{eq:optimal_policy_set_norm} is well-defined and nonempty. Moreover, for every $L^*\in \mathcal{L}_n^*$, the stationary linear policy $L^*x$ is optimal, i.e., $\mu^*(x)=L^*x$, and $(A+BL^*)$ is stable. 
\end{proposition}

As in the previous section, we say the system is stabilizable if there exists some $L\in \Re^{m\times n}$ such that $\Vert Lx\Vert\leq Nx$ for all $x$, and $(A+BL)$ is stable. The following corollary can be shown by using Prop.~\ref{prop:optimal_policy_norm}.

\begin{corollary}
    \label{coro:stabilizable_norm}
    The following two statements are equivalent:
    \begin{itemize}
    \item[(i)] $J^*(x)$ is finite for all $x\in X$.
    \item[(ii)] The system is stabilizable.
\end{itemize}
\end{corollary}

For this problem, the VI algorithm takes the form 
\begin{equation}
    \label{eq:vi_pos_norm}
    p_{k+1}=s+A'p_k-N'\Vert r+B'p_k\Vert_*.
\end{equation}
In what follows, we show that the VI algorithm \eqref{eq:vi_pos_norm} converges to $p^*$ starting from arbitrary $p_0\in \Re^n_+$.
\begin{proposition}
    \label{prop:vi_norm}
    Suppose that $J^*(x)$ is finite for all $x\in X$. Then the sequence $\{p_k\}$ generated by VI \eqref{eq:vi_pos_norm} with $p_0\in \Re^n_+$ converges to $p^*$.
\end{proposition}

Similar to the preceding section, PI and optimistic PI algorithms can also be defined for the norm bound problem. 
Assume that $J^*(x)$ is finite for all $x\in X$. Then in view of Prop.~\ref{prop:optimal_policy_norm}, there exists some $\mu^0(x)=L_0x$ such that $(A+BL_0)$ is stable. When applying PI algorithm, at a typical iteration $k$, given the policy $\mu^k(x)=L_kx$, the policy evaluation step takes the form 
\begin{equation}
    \label{eq:pi_eva_pos_norm}
    p_{\mu^k}'=(s+L_k'r)'(I-A-BL_k)^{-1}.
\end{equation}
The improved policy $\mu^{k+1}$ takes the form $\mu^{k+1}(x)=L_{k+1}x$ with $L_{k+1}$ given as 
\begin{equation}
    \label{eq:pi_imp_pos_norm}
    L_{k+1}=-w_kN
\end{equation}
where $w_k$ can be arbitrary element in $\mathcal{S}(r+B'p_{\mu^k})$, i.e., $w_k\in \mathcal{S}(r+B'p_{\mu^k})$. We have the following result for the PI algorithm.

\begin{proposition}
    \label{prop:pi_norm}
    Suppose that $J^*(x)$ is finite for all $x\in X$. Then there exists some $L_0$ such that $\mu^0(x)=L_0x$ is well-defined, i.e., $L_0x\in U(x)$ for all $x\in X$, and $(A+BL_0)$ is stable. Moreover, starting from $\mu^0(x)=L_0x$, the PI algorithm \eqref{eq:pi_eva_pos_norm} and \eqref{eq:pi_imp_pos_norm} is well-defined, i.e., $L_kx\in U(x)$ for all $x\in X$, $(A+BL_k)$ are stable, $k=1,2,\dots$, and $p_{\mu^k}$ converges to $p^*$, i.e., $p_{\mu^k}\to p^*$. In addition, the set of limit points of the sequence $\{L_k\}$ is a nonempty subset of $\mathcal{L}_n^*$.
\end{proposition}

Note that when the norm used in \eqref{eq:control_norm_b} is $1$ norm or $\infty$ norm, with a more elaborate definition of PI algorithm so that a unique policy is specified at each step, it is possible to establish convergence of PI after a finite number of iterations. However, we keep the result here general as such results would require tedious complications for the presentation. To see that the policies generated by PI need not converge, we give an example.

\begin{example}\label{eg:scalar_norm}
Consider again the scalar problem in Example~\ref{eg:bellman_nonunique_norm} with $N$ modified as $0.05$. One can verify that the Assumption~\ref{asm:2} holds. The plot of Bellman's equation is given in Fig.~\ref{eg:scalar_norm}. It can be seen that there is a unique solution to Bellman's equation, which is $p^*=10$. The optimal stationary policy $\mu^*(x)$ can be obtained via
$$\mu^*(x)\in\arg\min_{u\in U(x)}\big\{x-10u+10(0.9x+u)\big\}.$$
The expression subject to minimization is not a function of $u$. Therefore, every policy $\mu(x)=Lx$ with $Lx\in U(x)$ for all $x$ achieves the minimum thus optimal. As a result, the sequence $\{p_{\mu^k}\}$ converges after the first iteration, yet the sequence $\{L_k\}$ does not converge.

\begin{figure}[ht]
    \centering
    \includegraphics[width=0.9\linewidth]{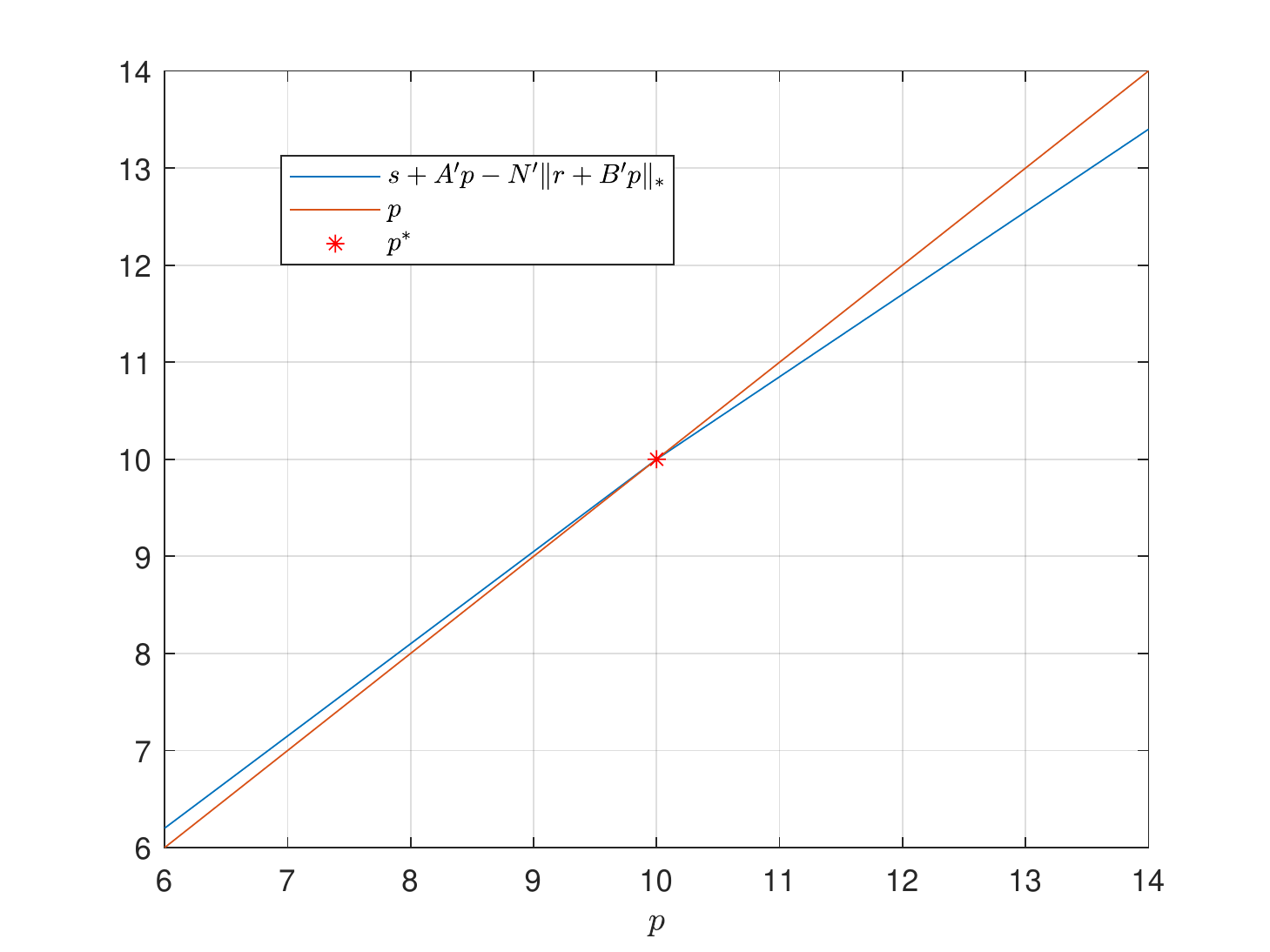}
    \caption{The Bellman curve and the solutions to the Bellman's equation for Example~\ref{eg:scalar_norm}. Note that $p^*$ is the intersecting point between $45^{\circ}$ and all the $L$-Bellman curves, namely all the linear policies are optimal.}
    \label{fig:scalar_norm}
\end{figure}
\end{example}

Similar to the absolute value bound problem, we can define an optimistic variant of PI for the norm bound problem. In particular, let $\{\ell_k\}$ be a sequence of positive integers. The optimistic PI starts with some $p_0$ such that 
\begin{equation}
    \label{eq:op_pi_initial_pos_norm}
    p_0\geq s+A'p_0-N'\Vert r+B'p_0\Vert_*.
\end{equation}
At a typical iteration $k$, given $p_k$, it defines a policy $\mu^k(x)=L_k(x)$ by
\begin{equation}
    \label{eq:op_pi_policy_norm}
    L_k=-w_kN,
\end{equation}
where $w_k\in \mathcal{S}(r+B'p_k)$ and obtain $p_{k+1}$ by
\begin{equation}
    \label{eq:op_pi_cost_norm}
    p_{k+1}'=p_k'(A+BL_k)^{\ell_k}+(s+L_k'r)'\sum_{i=0}^{\ell_k-1}(A+BL_k)^i.
\end{equation}
As in the preceding section, the optimistic PI is well-posed and the generated $p_k$ converges to $p^*$.

\begin{proposition}
    \label{prop:optimistic_pi_norm}
    Suppose that $J^*(x)$ is finite for all $x\in X$. Then there exists $p_0$ that satisfies the inequality \eqref{eq:op_pi_initial_pos_norm}, and the optimistic PI \eqref{eq:op_pi_policy_norm} and \eqref{eq:op_pi_cost_norm} is well-defined, i.e., $L_kx\in U(x)$ for all $x\in X$, $k=0,1,\dots$. Moreover, $(A+BL_k)$ are stable, $k=0,1,\dots$, and the generated $p_k$ converges to $p^*$, i.e., $p_k\to p^*$. In addition, the set of limit points of the sequence $\{L_k\}$ is a nonempty subset of $\mathcal{L}_n^*$.
\end{proposition}

As in the preceding case, the uniqueness result given in Prop.~\ref{prop:unique_norm} can also be used to formulate optimization programs for solving the problem. It takes the form 
\begin{subequations}
\label{eq:optimization}
    \begin{align}
	\max_{p,\gamma}& \quad 1'p\\
	\mathrm{s.\,t.} & \quad p\leq s+A'p-N'\gamma,\label{eq:optimization_eq}\\
	& \quad  \Vert r+B'p\Vert_*\leq \gamma, \label{eq:optimization_ineq}\\
	& \quad  p\in \Re^n_+,\;\gamma\in \Re_+.\label{eq:optimization_range}
	\end{align}
\end{subequations}
The exact expression of the constraint \eqref{eq:optimization_ineq} depends on the norm used in defining the control constraint set \eqref{eq:control_norm_b}. For example, when $1$ norm or $\infty$ norm is used, the constraint \eqref{eq:optimization_ineq} can be described by a collection of linear inequalities. When Euclidean norm is used, the inequality \eqref{eq:optimization_ineq} is quadratic in both $p$ and $\gamma$. Regardless of the type of the norm, the constraint \eqref{eq:optimization_ineq} specifies a closed convex set of $p$ and $\gamma$. The convexity allows us to assert the existence of some $p$ that attains the maximum in the optimization program \eqref{eq:optimization}. In fact, the convexity of the feasible region is inherited from the concavity of the Bellman operator in the cost functions for a variety of problems where the costs are accumulated over time; see \cite[p.~33]{bertsekas2022lessons}. This desirable property has been used to compute lower estimates of optimal cost functions of problems far beyond the scope of this work; see, e.g., \cite{rantzer2000piecewise}, \cite[Chapter~4]{boyd2013performance}.

For the optimization program, we have the following result. 

\begin{proposition}
    \label{prop:optimization}
    The following two statements are equivalent:
    \begin{itemize}
    \item[(i)] $J^*(x)$ is finite for all $x\in X$.
    \item[(ii)] The optimization program \eqref{eq:optimization} admits a finite optimal value $p^*$. Moreover, $J^*(x)=x'p^*$ for all $x\in X$.
\end{itemize}
\end{proposition}

\section{Proof of the Main Results}\label{sec:proof}
In this section, we provide proof arguments for the theories developed earlier. Some principle ideas in the proof are inspired by those presented in \cite{bertsekas2015value} and are shared by both classes of problems. As a result, we give a detailed analysis of the absolute value bound problem and highlight the different arguments used for the norm bound problem. 

\subsection{Analysis for the Absolute Value Bound Problem}\label{sec:ab_v_b_analysis}
For convenience, let us introduce some notations, which would highlight the structural properties of the problem. We focus on the subset $\mathcal{L}_e\subset \Re^{m\times n}$ defined as
$$\mathcal{L}_e=\{L\in \Re^{m\times n}\,|\,|L|\leq E\}.$$
Then for every $L\in \mathcal{L}_e$, the function $Lx$ is a policy, i.e., $Lx\in U(x)$ for all $x\in X$. Clearly, when $p^*$ is finite, we have $\mathcal{L}_e^*$ well-defined and $\mathcal{L}_e^*\subset \mathcal{L}_e$.

For every $L\in \mathcal{L}_e$, we introduce $L$-Bellman operator $G_L:\Re^n_+\mapsto\Re^n_+$ defined as
$$G_L(p)=s+L'r+(A+BL)'p.$$
Moreover, we introduce Bellman operator $G:\Re^n_+\mapsto\Re^n_+$ defined as
$$G(p)=s+A'p-E'|r+B'p|.$$
It can be verified that for all $x\in \Re^n_+$,
$$G(p)x=\inf_{L\in \mathcal{L}_e}G_L(p)x.$$
Besides, for every $p\in \Re_+^n$, there exists some $L\in \mathcal{L}_e$, such that 
\begin{equation}
    \label{eq:op_policy_G_L}
    G(p)=G_L(p). 
\end{equation}
In fact, one such matrix $L$ is given in explicit form as
\begin{equation}
    \label{eq:greedy_policy}
    L=-\begin{bmatrix}
\text{sign}(r_1+b_1'p)E_1\\
\vdots\\
\text{sign}(r_m+b_m'p)E_m
\end{bmatrix};
\end{equation}
cf. Eq.~\eqref{eq:optimal_policy_bar}. The $\ell$-fold compositions of $G_L$ and $G$ are denoted as $G_L^\ell$ and $G^\ell$, respectively, with $G^0_L(p)=G^0(p)=p$ by convention. 

As their names suggest, those are the ordinary Bellman operators restricted to the linear functions $J(x)=x'p$ and the linear policies $\mu(x)=Lx$. As a result, they inherit the monotonicity property from the general Bellman operators. In particular, for $p,\Bar{p}\in \Re_n^+$, if $p\leq\Bar{p}$, then
$$G_L(p)\leq G_L(\Bar{p}),\;\forall L\in \mathcal{L}_e,\quad G(p)\leq G(\Bar{p}).$$
Moreover, the operator $G$ is continuous, i.e., $p_k\to p$ implies that $G(p_k)\to G(p)$.

Now we are ready to prove the results stated in the previous section.

\begin{proof}[Proof for Prop.~\ref{prop:unique}]
Suppose that $J^*(x)$ is finite for all $x\in X$. Consider the sequence of functions $\{J_k\}$ generated by VI [cf. Eq.~\eqref{eq:vi}] with $J_0\equiv0$. Straightforward calculation shows that $J_k(x)=x'p_k$ for $p_k\in \Re_+^n$, $p_{k+1}=G(p_k)$, and $p_0=0$. In particular, we have
\begin{equation}
\label{eq:unique_p_1_ineq}
    p_1=s+A'p_0-E'|r+B'p_0|=s-E'|r|\geq 0
\end{equation}
by \eqref{eq:asum} in Assumption~\ref{asm:1}. Then by the monotonicity of $G$, we have that $\{p_k\}$ monotonically increasing and thus nonnegative, therefore 
\begin{equation}
    \label{eq:unique_p_k_ineq}
    \begin{aligned}
    p_{k+1}=&s+A'p_k-E'|r+B'p_k|\\
    \geq&s+A'p_k-E'|r|-E'|B'|p_k\\
    =&s-E'|r|+(A'-E'|B'|)p_k,
\end{aligned}
\end{equation}
where the first inequality is due to triangular inequality, $E\geq0$ and $p_k\geq 0$. Together with \eqref{eq:unique_p_1_ineq}, this yields
$$p_{k}'\geq(s'-|r'|E)\sum_{i=0}^{k-1}(A-|B|E)^i.$$
Due to \eqref{eq:asum_ob}, we have $p_n>0$. Since $J^*(x)$ is finite, then there exists some $p^*$ such that $p_k\to p^*$. Since $p_{k+1}=G(p_k)$, taking limits on both sides yields $p^*=G(p^*)$, or equivalently,
$$J_\infty(x)=\inf_{u\in U(x)}\big\{g(x,u)+J_\infty\big(f(x,u)\big)\big\},$$
where $J_\infty(x)=x'p^*$. By Prop.~\ref{prop:classical}(a), we have that $J_\infty\geq J^*$. However, $J_\infty(x)$ is the limit of a monotone sequence $x'p_k$, which is upper bounded by $J^*(x)$. Therefore, $J^*(x)=x'p^*$. Let $\Bar{p}$ be a solution to the equation $p=G(p)$. Then by Prop.~\ref{prop:classical}(a), $p^*\leq \Bar{p}$. In addition, there exists some positive $c>1$ such that $\Bar{p}\leq cp^*$ as $p^*\geq p_n>0$. Consider the sequences $\{\Tilde{p}_k\}$ and $\{\Bar{p}_k\}$ generated by VI with $\Tilde{p}_0=cp^*$ and $\Bar{p}_0=\Bar{p}$. Then by the monotonicity of $G$, $\Bar{p}_k\leq \Tilde{p}_k$. However, $\Bar{p}_k=\Bar{p}$ as $\Bar{p}_k=G(\Bar{p}_k)$ and $\Tilde{p}_k\to p^*$ due to Prop.~\ref{prop:classical}(d). Therefore, the solution to $p=G(p)$ is unique. 

Conversely, let $p^*$ be the unique solution of $p=G(p)$. Then $J^*(x)\leq x'p^*$ for all $x\in X$ by Prop.~\ref{prop:classical}(a).
\end{proof}

\begin{proof}[Proof of Prop.~\ref{prop:optimal_policy}]
    Since $J^*(x)$ is finite for all $x\in X$, then by Prop.~\ref{prop:unique}, $J^*(x)=x'p^*$ where $p^*$ is the unique solution to $p=G(p)$ within $\Re^n_+$. Then the set $\mathcal{L}_e^*$ given in \eqref{eq:optimal_policy_bar} is well-defined, nonempty, and is a subset of $\mathcal{L}_e$. In fact, we have $G_{L^*}(p^*)=p^*$ for every $L^*\in \mathcal{L}_e^*$. Define $\mu^*(x)=L^*x$. Then we have that 
    $$\mu^*(x)\in\arg\min_{u\in U(x)}\big\{s'x+r'u+J^*\big(f(x,u)\big)\big\},\quad \forall x\in X.$$
    By Prop.~\ref{prop:classical}(c), $\mu^*$ is optimal, i.e.,
    \begin{equation}
        \label{eq:stable_cost}
        \lim_{\ell\to\infty}(s'+r'L^*)\sum_{i=0}^{\ell}(A+BL^*)^ix=J^*(x)<\infty,\quad \forall x\in X=\Re_+^n.
    \end{equation}
    The definition of the set $\mathcal{L}_e$ yields that $(A+BL^*)\geq (A-|B|E)\geq0$, which implies that
    $$(A+BL^*)^i\geq (A-|B|E)^i\geq0,\quad i=1,2,\dots.$$
    Similarly, $s+(L^*)'r\geq (s- E'|r|)\geq0$. As a result, we have that
    \begin{equation}
    \label{eq:stable_positive}
        (s'+r'L^*)\sum_{i=0}^{n-1}(A+BL^*)^i\geq (s'- |r'|E)\sum_{i=0}^{n-1}(A-|B|E)^i>0.
    \end{equation}
    By rewriting the cost \eqref{eq:stable_cost}, we have
    \begin{align*}
        &\lim_{\ell\to\infty}(s'+r'L^*)\sum_{i=0}^{\ell}(A+BL^*)^ix\\
        =&(s'+r'L^*)\sum_{i=0}^{n-1}(A+BL^*)^i\lim_{\ell\to\infty}\sum_{j=0}^{\ell}(A+BL^*)^{nj}x
    \end{align*}
    where the equality is justified as the limit of a subsequence of a convergent sequence equals the limit of the sequence. In view of \eqref{eq:stable_cost}, \eqref{eq:stable_positive}, and Prop.~\ref{prop:f_p}, the largest absolute value of the eigenvalues of $(A+BL^*)^n$ is less than $1$, namely $(A+BL^*)$ is stable.
\end{proof}

\begin{proof}[Proof of Corollary~\ref{coro:stabilizable}]
    Prop.~\ref{prop:optimal_policy} shows that (i) implies (ii). To see it holds conversely, let $\mu(x)=Lx$ where $L\in \mathcal{L}_e$ and $(A+BL)$ is stable. Then
    $$J_{\mu}(x)=\lim_{\ell\to\infty}(s'+r'L)\sum_{i=0}^{\ell}(A+BL)^ix<\infty,\quad \forall x\in X=\Re_+^n.$$
    As a result, we have 
    \begin{equation}
    \label{eq:coro1}
        \inf_{u\in U(x)}\big\{g(x,u)+J_{\mu}\big(f(x,u)\big)\big\}\leq g\big(x,\mu(x)\big)+J_{\mu}\Big(f\big(x,\mu(x)\big)\Big)=J_{\mu}(x),
    \end{equation}
    where the equality is due to Prop.~1(b). Putting the first and the last term of Eq.~\eqref{eq:coro1} and in view of Prop.~1(a), we have $J^*(x)\leq J_\mu(x)<\infty$ for all $x$.
\end{proof}

\begin{proof}[Proof of Prop.~\ref{prop:vi}]
    Let $\{p_k\}$ be a sequence generated by VI with $p_0\in \Re^n_+$. Then there exists some $c>1$ such that $p_0\leq cp^*$, where $p^*$ fulfills that $J^*(x)=x'p^*$. Consider the sequences $\{\underline{p}_k\}$ and $\{\overline{p}_k\}$ generated by VI with $\underline{p}_0=0$ and $\overline{p}_0=cp^*$, respectively. Then $\underline{p}_0\leq p_0\leq \overline{p}_0$. By the monotonicity of $G$, we have 
    $$\underline{p}_k\leq p_k\leq \overline{p}_k,\quad k=0,1,\dots.$$
    However, as is argued in the proof for Prop.~\ref{prop:unique}, $\underline{p}_k\to p^*$ and $\overline{p}_k\to p^*$. As a result, we have that $p_k\to p^*$.
\end{proof}

\begin{proof}[Proof of Prop.~\ref{prop:pi}]
    First, we note that the sequence of functions $\{J_{\mu^k}\}$ generated by PI for addressing the generic nonnegative cost problems are monotonically decreasing. This can be shown through the following inequalities:
    \begin{align*}
        J_{\mu^k}(x)=&g\big(x,\mu^k(x)\big)+J_{\mu^k}\Big(f\big(x,\mu(x)\big)\Big)\\
        \geq &\inf_{u\in U(x)}\big\{g(x,u)+J_{\mu^k}\big(f(x,u)\big)\big\}\\
        = &g\big(x,\mu^{k+1}(x)\big)+J_{\mu^k}\Big(f\big(x,\mu(x)\big)\Big),
    \end{align*}
    where the first equality is due to Prop.~\ref{prop:classical}(b) and the second equality is due to the definition of $\mu^{k+1}$; cf. Eq.~\eqref{eq:policy_im}. Then applying Prop.~\ref{prop:classical}(b) with $J_{\mu^k}$ in place of $\hat{J}$, we have that $J_{\mu^{k+1}}\leq J_{\mu^k}$.

    When applying PI to the positive systems, the existence of $L_0\in\mathcal{L}_e$ such that $(A+BL_0)$ is stable is assured in view of the assumption $J^*<\infty$ and Prop.~\ref{prop:optimal_policy}. Therefore, the cost function of $\mu^0$ is given as $J_{\mu^0}(x)=x'p_{\mu^0}$ with $p_{\mu^0}=(s+L_0'r)'(I-A-BL_0)^{-1}$. To see that the PI algorithm \eqref{eq:pi_eva_pos} and \eqref{eq:pi_imp_pos} is well-posed, we apply induction. Assume that $(A+BL_k)$ is stable and $J_{\mu^k}(x)=x'p_{\mu^k}$. By the generic cost improvement property $J_{\mu^{k+1}}\leq J_{\mu^k}$, we have that 
    $$J_{\mu^{k+1}}(x)=\lim_{\ell\to\infty}(s'+r'L_{k+1})\sum_{i=0}^{\ell}(A+BL_{k+1})^ix\leq J_{\mu^k}(x)<\infty,\quad \forall x\in X=\Re_+^n.$$
    The definition of the set $\mathcal{L}_e$ yields that $(A+BL_{k+1})\geq0$. Then by the same arguments applied in the proof of Prop.~\ref{prop:optimal_policy}, we have that $(A+BL_{k+1})$ is stable. Therefore, $J_{\mu^k}(x)=x'p_{\mu^k}$ for some $p_k\in \Re^n_+$, $k=0,1,\dots$, and $J_{\mu^{k+1}}\leq J_{\mu^k}$ implies that $\{p_{\mu^k}\}$ is decreasing, which is convergent as it is lower bounded by $p^*$.

    To see that the PI converges to the optimal after a finite number of iterations, we note that every $L_k$, $k=1,2,\dots$, must take the form \eqref{eq:greedy_policy}, which can be as many as $2^m$ in the event that $E_i\neq 0$ for all $i=1,\dots,m$. If $p_{\mu^{k+1}}\neq p_{\mu^k}$, then clearly $L_{k+1}\neq L_k$. If $p_{\mu^{k+1}}= p_{\mu^k}$, then we have that 
    \begin{equation}
        \label{eq:pi_imp_proof}
        G_{L_{k+1}}(p_{\mu^{k+1}})=G_{L_{k+1}}(p_{\mu^k})=G(p_{\mu^k})=G(p_{\mu^{k+1}})\leq p_{\mu^k}=p_{\mu^{k+1}}.
    \end{equation}
    However, in view of Prop.~\ref{prop:classical}(b), the equality $G_{L_{k+1}}(p_{\mu^{k+1}})=p_{\mu^{k+1}}$ holds, which combined with Eq.~\eqref{eq:pi_imp_proof}, implies that
    $$G(p_{\mu^{k+1}})=p_{\mu^{k+1}}.$$
    Since $G(p)=p$ admits a unique solution within $\Re^n_+$ due to Prop.~\ref{prop:unique}, then $p_{\mu^{k+1}}=p^*$. Moreover, all the subsequent policies $\mu^{\ell}$, $\ell=k+1,k+2,\dots$, and the equalities $p_{\mu^{\ell+1}}=p_{\mu^{\ell}}$ hold. This, in turn, implies that for the sequence $\{p_{\mu^k}\}_{k=0}^{\Bar{k}}$ where $\Bar{k}=\min\{\ell\,|\,p_{\mu^{\ell+1}}= p_{\mu^\ell}\}$, we have that $p_{\mu^{k+1}}\neq p_{\mu^k}$, so the policies generated in PI are not repeated, until after some $\Bar{k}$. Since there can be at most $2^m$ generated policies, the PI converges finitely.   
\end{proof}

\begin{proof}[Proof of Prop.~\ref{prop:optimistic_pi}]
    The existence of $p_0\in \Re^n_+$ that fulfills the inequality $p_0\geq G(p_0)$ is ensured in view of Prop.~\ref{prop:unique} as $J^*<\infty$. The well-posedness of optimistic PI can be shown in a similar way in which the well-posedness of PI is established. Therefore, we omit this part of the proof.

    For the convergence of optimisitic PI, we consider an auxiliary sequence $\{\overline{p}_k\}$ with $\overline{p}_0=p_0$ and $\overline{p}_{k+1}=G(\overline{p}_k)$. Those two sequences define sequences of functions $\{J_k\}$ and $\overline{J}_k$ via $J_k(x)=x'p_k$ and $\overline{J}_k(x)=x'\overline{p}_k$. In what follows, we will show the inequalities 
    \begin{equation}
        \label{eq:op_pi_inequality}
        p^*\leq p_k\leq \overline{p}_k,\quad k=0,1,\dots
    \end{equation}
    by induction. Since $G(p_0)\leq p_0$, then by Prop.~\ref{prop:classical}(a) $p^*\leq p_0= \overline{p}_0$. Moreover, $G(p_0)\leq p_0$ and $G(\overline{p}_0)\leq \overline{p}_0$ Suppose $p^*\leq p_k\leq \overline{p}_k$, $G(p_k)\leq p_k$ and $G(\overline{p}_k)\leq \overline{p}_k$. By monotonicity of $G$, $p_k\leq \overline{p}_k$ and $G(\overline{p}_k)\leq \overline{p}_k$ imply that $G(p_k)\leq G(\overline{p}_k)=\overline{p}_{k+1}$ and $G^2(\overline{p}_k)\leq G(\overline{p}_k)=\overline{p}_{k+1}$, respectively. However, $G^2(\overline{p}_k)=G\big(G(\overline{p}_k)\big)=G(\overline{p}_{k+1})$, which yields that $G(\overline{p}_{k+1})\leq \overline{p}_{k+1}$. By definition of $L_{k+1}$, we also have that $G_{L_{k+1}}(p_k)=G(p_k)\leq p_k$. Due to the monotonicity of $G_{L_{k+1}}$, we have that 
    $$G_{L_{k+1}}^{\ell_k+1}(p_k)\leq G_{L_{k+1}}^{\ell_k}(p_k)\leq G_{L_{k+1}}(p_k)=G(p_k)\leq p_k.$$
    In view of the definition of $p_{k+1}=G_{L_{k+1}}^{\ell_k}(p_k)$ [cf. Eq.~\eqref{eq:op_pi_cost}] and the operator $G$, we have that 
    $$G(p_{k+1})\leq G_{L_{k+1}}(p_{k+1})\leq p_{k+1}\leq G(p_k)\leq p_k.$$
    Moreover, the inequalities $G_{L_{k+1}}(p_{k+1})\leq p_{k+1}$ implies that $J_{\mu}(x)\leq x'p_{k+1}$ where $\mu(x)=L_{k+1}x$ due to Prop.~\ref{prop:classical}(b). As a result, $x'p_{k+1}\geq x'p^*$, or equivalently, $p_{k+1}\geq p^*$.\footnote{Note that for the generic nonnegative cost problems, let $J_0=\overline{J}_0$ such that inequality
    $$\inf_{u\in U(x)}\big\{g(x,u)+J_0\big(f(x,u)\big)\big\}\leq J_0(x),\quad \forall x\in X$$
    holds. The existence of such function is no issue as one may set $J_0=\infty$. Let $\{J_k\}$ and $\{\overline{J}_k\}$ be the sequences generated by optimistic PI and VI respectively, then the inequality 
    $$J^*\leq J_k\leq \overline{J}_k,\quad k=0,1,\dots,$$
    which is the generalization of Eq.~\eqref{eq:op_pi_inequality}, remains valid, and can be proved using nearly identical arguments applied above. This is expected, as this part of our proof relies entirely on the monotonicity of $G$ and $G_L$, which are inherited from the generic Bellman operators for nonnegative cost problems.}

    For the convergence of $\{p_k\}$, by Prop.~\ref{prop:vi}, $\overline{p}_k\to p^*$. Then in view of the inequality \eqref{eq:op_pi_inequality}, we have that $p_k\to p^*$. To see the finite convergence towards the optimal policy, let us define index set $\mathcal{I}$ as 
    $$\mathcal{I}=\{i\,|\,r_i+b_i'p^*\neq0\}.$$
    Since $p_k\to p^*$, $r_i+b_i'p_k\to r_i+b_i'p^*$, $i=1,\dots,n$. As a result, there exists some $\Bar{k}$ such that $\text{sign}(r_i+b_i'p_k)=\text{sign}(r_i+b_i'p^*)$ for all $k\geq \Bar{k}$ and $i\in \mathcal{I}$. Then it can be seen that the subsequent policies are all optimal, i.e., $L_k\in\mathcal{L}_e^*$ for all $k\geq \Bar{k}$.
\end{proof}

\begin{proof}[Proof of Prop.~\ref{prop:linear}]
    First, we note that the constraints \eqref{eq:linear_eq}, \eqref{eq:linear_ineq}, and \eqref{eq:linear_range} are equivalent to $p\leq G(p)$ and $p\in \Re^n_+$. As a result, the sequence $\{G^k(p)\}$ is monotonically increasing. Moreover, if $\hat{p}\in \Re_+^n$ attains the maximum of the linear program \eqref{eq:linear}, then $\hat{J}(x)=x'\hat{p}$ solves Bellman's equation, i.e.,
    $$\hat{J}(x)=\inf_{u\in U(x)}\big\{g(x,u)+\hat{J}\big(f(x,u)\big)\big\},\quad \forall x\in X.$$

    Suppose that $J^*<\infty$. Then by Prop.~\ref{prop:unique}, $J^*(x)=x'p^*$ with $p^*\in \Re^n_+$. For any $p\in \Re^n_+$ such that $p\leq G(p)$, in view of Prop.~\ref{prop:vi}, $G^k(p)\to p^*$. Since $\{G^k(p)\}$ is monotonically increasing, then $p\leq p^*$. Therefore, every feasible $p$ of the linear program is bounded above by $p^*$.

    Suppose that $J^*(\hat{x})=\infty$ for some $\hat{x}\in X$. Assume that the linear program \eqref{eq:linear} is bounded. Let $\hat{p}\in \Re_+^n$ attain the maximum of the linear program \eqref{eq:linear} (the existence of such $\hat{p}$ is assured, see, e.g., \cite{bertsimas1997introduction}). As discussed above, $\hat{J}(x)=x'\hat{p}$ solves Bellman's equation. This implies that $\hat{x}'\hat{p}\geq J^*(\hat{x})$, which is a contradiction.  
\end{proof}

\subsection{Analysis for the Norm Bound Problem}\label{sec:norm_b_analysis}
Similar to the notations introduced in Section~\ref{sec:ab_v_b_analysis}, we introduce the subset $\mathcal{L}_n\subset \Re^{m\times n}$ defined as
$$\mathcal{L}_n=\{L\in \Re^{m\times n}\,|\,L=-wN,\;\Vert w\Vert\leq 1\}.$$ 
In addition, for every $L\in \mathcal{L}_n$, we introduce $L$-Bellman operator $F_L:\Re^n_+\mapsto\Re^n_+$ defined as
$$F_L(p)=s+L'r+(A+BL)'p.$$
Moreover, we introduce Bellman operator $F:\Re^n_+\mapsto\Re^n_+$ defined as
$$F(p)=s+A'p-N'\Vert r+B'p\Vert_*.$$
The $\ell$-fold compositions of $F_L$ and $F$ are denoted as $F_L^\ell$ and $F^\ell$, respectively, with $F^0_L(p)=F^0(p)=p$ by convention. Identical observations following Assumption~\ref{asm:1} with regards to $G_L$ and $G$ can be made for $F_L$ and $F$, respectively. Now we are ready to prove the main results for this class of problems.

\begin{proof}[Proof for Prop.~\ref{prop:unique_norm}]
The proof can be obtained by using nearly identical arguments in the proof for Prop.~\ref{prop:unique} with $F$ in place of $G$. Here we provide the algebraic manipulations used to show that $p^*>0$. 

For the VI starting from $p_0=0$, we have that 
$$p_{1}=s+A'p_0-N'\Vert r+B'p_0\Vert_*=s-N'\Vert r\Vert_*\geq 0,$$
and
\begin{align*}
    p_{k+1}=&s+A'p_k-N'\Vert r+B'p_k\Vert_*\\
    \geq& s+A'p_k-N'\Vert r\Vert_*-N'[\Vert B_1'\Vert_*\dots \Vert B_n'\Vert_*]p_k\\
    =& s-N'\Vert r\Vert_*+\big(A'-N'[\Vert B_1'\Vert_*\dots \Vert B_n'\Vert_*]\big)p_k.
\end{align*}
The inequality is used to show that $p_n>0$, which yields $p^*>0$.
\end{proof}

\begin{proof}[Proof of Prop.~\ref{prop:optimal_policy_norm}]
    The proof can be obtained by using nearly identical arguments in the proof for Prop.~\ref{prop:optimal_policy} with $G$, $G_{L^*}$, $\mathcal{L}_e$, and $\mathcal{L}^*_e$ replaced by their counterparts. As in the proof above, we provide the algebraic steps leading to the stability arguments.

    The definition of the set $\mathcal{L}_n$ yields that 
    $$A+BL^*\geq \big(A-[\Vert B_1'\Vert_*\dots \Vert B_n'\Vert_*]'N\big)\geq0$$
    which implies that
    $$(A+BL^*)^i\geq \big(A-[\Vert B_1'\Vert_*\dots \Vert B_n'\Vert_*]'N\big)^i\geq0,\quad i=1,2,\dots.$$
    Similarly, $s+(L^*)'r\geq (s- N'\Vert r\Vert_*)\geq0$. As a result, we have that
    \begin{align*}
        &(s'+r'L^*)\sum_{i=0}^{n-1}(A+BL^*)^i\\
        \geq& (s'-\Vert r\Vert_*N')\sum_{i=0}^{n-1}\big(A- [\Vert B_1'\Vert_*\dots \Vert B_n'\Vert_*]'N\big)^i>0,
    \end{align*}
    which leads to the stability of $(A+BL^*).$
\end{proof}

\begin{proof}[Proof of Corollary~\ref{coro:stabilizable_norm}]
    Identical arguments in the proof of Corollary~\ref{coro:stabilizable} can be used here once $\mathcal{L}_e$ is replaced by $\mathcal{L}_n$. 
\end{proof}

\begin{proof}[Proof of Prop.~\ref{prop:vi_norm}]
    The proof can be obtained by using identical arguments in the proof for Prop.~\ref{prop:vi} with $F$ in place of $G$. 
\end{proof}

\begin{proof}[Proof of Prop.~\ref{prop:pi_norm}]
    Convergence of the sequence $\{p_{\mu^k}\}$, the well-posedness of the PI algorithm, and the stability of matrices $(A+BL_k)$ can be shown by using the identical arguments of those in the related parts of the proof of Prop.~\ref{prop:pi} with $\mathcal{L}_n$ in place of $\mathcal{L}_e$.  

    To see that $\{p_{\mu^k}\}$ converges to $p^*$, we consider the auxiliary sequence $\{\Bar{p}_k\}$ with $\Bar{p}_0=p_{\mu^0}$, and $\Bar{p}_{k+1}=F(\Bar{p}_k)$. We will show by induction that
    $$p_{\mu^k}\leq \Bar{p}_k,\quad k=0,1,\dots.$$
    The inequality holds for $0$ in view of the definition of $\Bar{p}_0$. Suppose that $p_{\mu^k}\leq \Bar{p}_k$ with $\mu^k(x)=L_kx$. Then by the monotonicity of $F$, we have that 
    \begin{equation}
        \label{eq:pi_norm_ineq4}
        F(p_{\mu^k})\leq F(\Bar{p}_k)=\Bar{p}_{k+1}.
    \end{equation}
    In addition, we have the inequality 
    \begin{equation}
        \label{eq:pi_norm_ineq}
        F_{L_{k+1}}(p_{\mu^k})=F(p_{\mu^k})\leq F_{L_k}(p_{\mu^k})=p_{\mu^k},
    \end{equation}
    where $\mu^{k+1}(x)=L_{k+1}x$, the first equality is due to the definition of $L_{k+1}$, cf. \eqref{eq:pi_imp_pos_norm}, the inequality is due to the relations between $F$ and $F_{L_k}$, and the second equality is due to Porp.~\ref{prop:classical}(b). Putting the first and the last term of \eqref{eq:pi_norm_ineq} together gives
    \begin{equation}
        \label{eq:pi_norm_ineq_1}
        F_{L_{k+1}}(p_{\mu^k})\leq p_{\mu^k}.
    \end{equation}
    Applying on both sides of $F_{L_{k+1}}$ and in view of the monotonicity of $F_{L_{k+1}}$, we have 
    $$F_{L_{k+1}}\big(F_{L_{k+1}}(p_{\mu^k})\big)\leq F_{L_{k+1}}(p_{\mu^k}),$$
    or equivalently,
    $$s'x+r'L_{k+1}x+\hat{J}\big((A+BL_{k+1})x\big)\leq \hat{J}(x),\quad \forall x\in X,$$
    with $\hat{J}(x)=x'F_{L_{k+1}}(p_{\mu^k})$. Then in view of Prop.~\ref{prop:classical}(b), we have that
    \begin{equation}
        \label{eq:pi_norm_ineq2}
        p_{\mu^{k+1}}\leq F_{L_{k+1}}(p_{\mu^k}).
    \end{equation}
    Combining \eqref{eq:pi_norm_ineq4} and \eqref{eq:pi_norm_ineq2} yields that $p_{\mu^{k+1}}\leq \Bar{p}_{k+1}$. Since $\Bar{p}_{k}\to p^*$ by Prop.~\ref{prop:vi_norm} and $p_{\mu^{k}}\geq p^*$ for all $k$ by definition, we have that $p_{\mu^{k}}\to p^*$.

    By the definition of $L_k$, cf. \eqref{eq:pi_imp_pos_norm}, the sequence $\{L_k\}$ is bounded. Therefore the set of its limit points is nonempty. If $r+B'p^*=0$, the claim that its limit points belong to $\mathcal{S}(r+B'p^*)$ holds trivially. Otherwise, recall that $L_{k+1}=-w_kN$ with $w_k\in \mathcal{S}(r+B'p_{\mu^k})$, and we will show that the limit points of $\{w_k\}$ belong to $\mathcal{S}(r+B'p^*)$. Let us consider a subsequence of $w_k$ indexed by $\mathcal{K}$ such that $\{w_k\}_{k\in \mathcal{K}}$ is convergent with limit $\Bar{w}$. Since $r+B'p^*\neq0$, then $r+B'p_{\mu^k}\neq 0$ for sufficiently large $k$. As a result, $\Vert w_k\Vert =1$ for large $k$, which implies that $\Vert \Bar{w}\Vert=1$. Since $w_k\in \mathcal{S}(r+B'p_{\mu^k})$, we have that 
    \begin{equation}
        \label{eq:dual_norm_eq}
        \Vert r+B'p_{\mu^k}\Vert_*=w_k'(r+B'p_{\mu^k}).
    \end{equation}
    Since $p_{\mu^k}\to p^*$, taking limits on both sides of $k\in \mathcal{K}$ and in view of the continuity of dual norm, we have that
    $$\Vert r+B'p^*\Vert_*=\Bar{w}'(r+B'p^*).$$
    Together with $\Vert \Bar{w}\Vert =1$, this implies that $\Bar{w}\in \mathcal{S}(r+B'p^*)$, which concludes the proof.
\end{proof}

\begin{proof}[Proof of Prop.~\ref{prop:optimistic_pi_norm}]
    The existence of $p_0\in \Re^n_+$ that fulfills the inequality $p_0\geq F(p_0)$ is ensured in view of Prop.~\ref{prop:unique_norm} as $J^*<\infty$. The well-posedness of optimistic PI can be shown in a similar way in which the well-posedness of PI is established. Convergence of the sequence $\{p_k\}$ and the stability of matrices $(A+BL_k)$ can be shown by using the identical arguments of those in the related parts of the proof of Prop.~\ref{prop:optimistic_pi} with $F$ and $F_{L_{k+1}}$ in place of $G$ and $G_{L_{k+1}}$, respectively. 
    
    The property of the limit points of $\{L_k\}$ can be shown by using identical arguments as those in the corresponding part of the proof for Prop.~\ref{prop:pi_norm} with $p_k$ in place of $p_{\mu^k}$.
\end{proof}

\begin{proof}[Proof of Prop.~\ref{prop:optimization}]
    The majority part of proof can be obtained by using identical arguments as those in the proof of Prop.~\ref{prop:linear} with $F$ in place of $G$. 
    
    Some subtle difference in the arguments is required when $J^*(\hat{x})=\infty$ for some $\hat{x}\in X$. Assume that the optimizaiton program \eqref{eq:optimization} is bounded. Then we may add constraint $\gamma\leq \Bar{\gamma}$ for sufficiently large $\Bar{\gamma}$ without loss of generality. As discussed earlier, the constraint \eqref{eq:optimization_ineq} specifies a closed convex set. Then together with other constraints and the redundant constraint $\gamma\leq \Bar{\gamma}$, the feasible region of \eqref{eq:optimization} is nonempty (containing at least $p=0,\gamma =\Vert r\Vert_*$), convex, closed, and bounded. Then it admits a solution $\hat{p}$. 
\end{proof}

\section{Concluding Remarks}
We studied optimal control problems with positive linear systems, linear stage costs, and control constraints. We addressed two types of control constraints: one involving bounds on the elements of control and the other involving bounds on the norms of control. In both cases, we applied the generic DP theory for nonnegative cost problems and a form of the Frobenius-Perron theorem for the analysis. In particular, we provided conditions under which the solutions to the associated Bellman's equations are unique and analyzed the stability of the optimal policies. Based on these results, we also showed the convergence of VI, PI, and optimistic PI applied to the problems, as well as the equivalence between finite optimal costs and boundedness of the associated optimization programs.

Much of our analysis relied upon the assumption that the optimal costs are finite for all states. This assumption is equivalent to the stabilizable condition adapted to our context. If one assumes that the state can reach the origin within a finite number of steps, then an alternative framework developed in \cite{bertsekas2015value} can be brought to bear. 

Apart from the tools applied in our study, our results can also be interpreted via the abstract DP framework involving the concept of \emph{regularity} of policies, which means stabilizable policy within our context. The related theory is referred to as the \emph{semicontractive DP}, given in \cite[Chapter~3]{bertsekas2022abstract}. The name reflects the fact that in the problems studied via this theory, some policies are `well-behaved' (in certain sense) while others do not. Much of the theory reported in this work can also be developed within the semicontractive DP framework, and a result that is particularly relevant is \cite[Prop.~3.3.1]{bertsekas2022abstract}. Another abstract framework within which our results can be interpreted suitably is the recent work \cite{sargent2023completely}, where various forms of the DP algorithms are characterized by partial orders. These frameworks are elegant and can be used for unifying the results of this work with other classical results, such as those in the linear quadratic problems. On the other hand, it seems most convenient to take the approach of this work if one wishes to simplify the original infinite-dimensional problem to a finite one.  

\section*{Acknowledgement}
The authors would like to thank Prof. John Stachurski for the helpful comments and for bringing up the connection between the observability condition and the irreducibility property discussed in the Appendix.

\appendix

\section{Discussions on Observability}
In Assumptions~\ref{asm:1} and \ref{asm:2}, we require the inequalities
\begin{equation}
    \label{eq:app_ob_ab}
    (s'- |r'|E)\sum_{i=0}^{n-1}(A-|B|E)^i>0
\end{equation}
and 
$$(s'-\Vert r\Vert_*N')\sum_{i=0}^{n-1}\big(A- [\Vert B_1'\Vert_*\dots \Vert B_n'\Vert_*]'N\big)^i>0$$
hold, respectively. In the subsequent developments, these inequalities play the role that the observability condition in the classical linear quadratic problem acts, see, e.g., \cite[Definition~3.1.1]{bertsekas2017dynamic}, hence the name. To highlight the similarity even further, we provide the following result.

\begin{proposition}\label{prop:ob}
Let $v\in \Re^n_+$ and $M\in \Re^{n\times n}$ such that $M\geq 0$. The inequality 
$$v'\sum_{i=0}^{n-1}M^i>0$$
holds if and only if there exists some positive integer $\ell\geq1$ such that
\begin{equation}
        \label{eq:ob}
        v'\sum_{j=0}^{\ell-1}M^j>0.
    \end{equation}
\end{proposition}
\begin{proof}
    The only if part is obvious. To see that it holds conversely, suppose that the $k$th element of $v'\sum_{i=0}^{n-1}M^i$ is zero. We will show that the strict inequality in \eqref{eq:ob} does not hold for all $\ell$. Indeed, as $v\geq 0$ and $M\geq 0$, the $k$th elements of row vectors $v'M^i$, $i=0,\dots,n-1$, must all be zero. For every $j>n-1$, the matrix $M^j$ can be expressed as linear combinations (with real coefficients) of $M^i$, $i=0,\dots,n-1$, due to Cayley–Hamilton theorem (see, e.g., \cite[Theorem~2.4.3.2]{horn2012matrix}). As a result, the $k$th element of $v'M^j$ is also zero, which concludes the proof.
\end{proof}

Let us focus on the absolute value bound problem for the subsequent discussion, and entirely similar conclusions can be drawn for the norm bound problem. Suppose that $s>E'|r|$ (instead of $s\geq E'|r|$), then the stage cost $g(x,u)>0$ for all $x\neq0$. All the theories developed earlier still hold and some of the proof arguments can be simplified. Clearly, this is a special case in which the inequality
$$(s'- |r'|E)\sum_{i=0}^{n-1}(A-|B|E)^i>0$$
holds. The relations between the two conditions are reminiscent of those imposed on the cost coefficient in the linear quadratic problem, and $s>E'|r|$ corresponds to the coefficient matrix being positive definite.

The observability condition \eqref{eq:app_ob_ab} also has a close connection with the \emph{irreducibility} property of the matrix $(A-|B|E)$. In particular, the nonnegative matrix $(A-|B|E)$ is irreducible if and only if 
$$\sum_{i=0}^{n-1}(A-|B|E)^i>0,$$
cf. \cite[Lemma~8.4.1]{horn2012matrix}. As a result, a sufficient condition for the observability condition is that $(A-|B|E)$ is irreducible, and at least one element of $(s'- |r'|E)$ is nonzero. However, the irreducibility of $(A-|B|E)$ is not necessary. To see this, consider
$$s-E'|r|=\begin{bmatrix}
0 \\
1
\end{bmatrix},\quad A-|B|E=\begin{bmatrix}
1 & 0\\
1 & 0
\end{bmatrix}.$$
One can verify that $(A-|B|E)$ is not irreducible, yet the observability condition \eqref{eq:app_ob_ab} holds.
\bibliographystyle{alpha}
\bibliography{ref}
\end{document}